\documentclass[10pt]{article}

\usepackage[utf8]{inputenc}
\usepackage{amssymb, amscd,latexsym}
\usepackage{amsmath}
\usepackage{amsthm}
\usepackage[dvips]{graphicx}
\usepackage[all]{xy}
\usepackage{mathrsfs}
\usepackage{pstricks-add}
\usepackage{pst-all}
\usepackage{pst-poly} 
\usepackage{pst-poly} 
\usepackage{multido} 
\usepackage{calc} 

\newtheorem{thm}{Theorem}[section]
\newtheorem{cor}[thm]{Corollary}
\newtheorem{lem}[thm]{Lemma}
\newtheorem{prop}[thm]{Proposition}
\newtheorem{defn}[thm]{Definition}
\newtheorem{rem}[thm]{Remark}
\newtheorem{ex}[thm]{Example}

\def\cocoa{{\hbox{\rm C\kern-.13em o\kern-.07em C\kern-.13em o\kern-.15em A}}}

\def\sqr#1#2{{\vcenter{\hrule height.#2pt
        \hbox{\vrule width.#2pt height#1pt \kern#1pt
            \vrule width.#2pt}
        \hrule height.#2pt}}}

\def\depth{{\rm depth}\,}

\def\reg{{\rm reg\,}}
\def\regi{{\rm reg\kern-4pt-\kern-4ptindex}}

\def\restr{{\kern-1pt\restriction\kern-1pt}}

\def\NN{{\mathbb N}}

\def\ZZ{{\mathbb Z}}

\def\C{{\mathcal C}}

\begin{document}

\title{\bf Monomial ideals with 3-linear resolutions}
\author{Marcel Morales$^2$, Abbas Nasrollah Nejad$^1$, \\ Ali Akbar Yazdan Pour$^{1,2}$, Rashid Zaare-Nahandi$^1$\\
\small $^{1}$ Institute for Advanced Studies in Basic Sciences, P. O. Box 45195-1159, Zanjan, Iran\\
\small $^{2}$ Universit\'e de Grenoble I, Institut Fourier, Laboratoire de Math\'ematiques, France}
\date{}
\maketitle

\footnotetext{MSC 2010: Primary 13C14, 13D02; Secondary 13F55.\\
Key words: minimal free resolution, linear resolution, uniform clutter, monomial ideal, regularity.}

\begin{abstract}

In this paper, we study Cstelnuovo-Mumford regularity of square-free monomial ideals generated in degree $3$. We define some operations on the clutters associated to such ideals and prove that the regularity is conserved under these operations. We apply the operations to introduce some classes of ideals with linear resolutions and also show that any clutter corresponding to a triangulation of the sphere does not have linear resolution while any proper sub-clutter of it has a linear resolution.

\end{abstract}

\section{Introduction}
Let $S=K[x_1,\ldots,x_n]$ be the polynomial ring over a field $K$ and $I$ be a homogeneous ideal of $S$.
Computing the Castenuovo-Mumford regularity of $I$ or even proving that the ideal $I$ has linear resolutions is difficult in general. It is known that a monomial ideal has $d$-linear resolution if and only if its polarization, which is a square-free monomial ideal, has $d$-linear resolution. Therefore, classification of monomial ideals with linear resolution is equal to classification of square-free monomial ideals. In this subject, one of the fundamental results is the Eagon-Reiner theorem, which says that the Stanley-Reisner ideal of a simplicial complex has a linear resolution if and only if its Alexander dual is Cohen-Macaulay.

The problem of existing 2-linear resolution is completely solved by Fr\"oberg \cite{Fr} (See also \cite{Morales}). An ideals of $S$ generated by square-free monomials of degree 2 can be assumed as edge ideal of a graph. Fr\"oberg proved that the edge ideal of a finite simple graph $G$ has linear resolution if and only if the complementary graph $\bar{G}$ of $G$ is chordal, i.e., there is no any minimal induced cycle in $G$ of length more than 3.

Another combinatorial objects corresponding to square-free monomial ideals are clutters which are special cases of hypergraphs. Let $[n]=\{1, \ldots, n\}$. A clutter $\C$ on a vertex set $[n]$ is a set of subsets of $[n]$ (called circuits of $\C$) such that if $e_1$ and $e_2$ are distinct circuits, then $e_1 \nsubseteq e_2$.
A $d$-circuit is a circuit with $d$ vertices, and a clutter is called $d$-uniform if every circuit is a $d$-circuit. To a clutter $\C$ with circuits $\{e_1,\ldots,e_m\}$ the ideal generated by $X_{e_j}$ for all $j=1,\ldots, m$ is corresponded which is called circuit ideal of $\C$ and denoted by $I(\C)$. One says that a $d$-uniform clutter $\C$ has a linear resolution if the circuit ideal of the complimentary clutter $\bar{\C}$ has $d$-linear resolution. Trying to generalize similar result of Fr\"oberg for $d$-uniform clutters ($d>2$), several mathematicians including E. Emtander \cite{Emtander} and R. Woodroofe \cite{Wood1} have defined notion of chordal clutters and proved that any $d$-uniform chordal clutter has a linear resolution. These results are one-sided. That is, there are non-chordal $d$-uniform clutters with linear resolution.

In the present paper, we introduce some reduction processes on 3-uniform clutters which do not change regularity of the minimal resolution. Then a class of 3-uniform clutters which have linear resolution and a class of 3-uniform clutters which do not have linear resolution are constructed.

Some of the results of this paper have been conjectured after explicit computations performed by the computer algebra systems {\sc Singular}~\cite{Si} and \cocoa~\cite{coco}.

\section{Preliminaries}

Let $K$ be a field, $S=K[x_1,\ldots, x_n]$ be the polynomial ring over $K$  with the standard graduation and $\mathfrak{m} = (x_1, \ldots, x_n)$ be the irredundant maximal ideal of $S$.

We quote the following well-known results that will be used in this paper.

\begin{thm}[{Grothendieck, {\cite[Theorem 6.3]{Stan1}}}] \label{Grothendieck}
Let $M$ be a finitely generated $S$-module. Let $t = \depth(M)$ and $d = \dim(M)$. Then $H^i_\mathfrak{m}(M) \neq 0$ for $i = t$ and $i = d$, and $H^i_\mathfrak{m}(M) = 0$ for $i < t$ and $i > d$.
\end{thm}

\begin{cor}
Let $M$ be a finitely generated $S$-module. $M$ is Cohen-Macaulay if and only if $H^i_\mathfrak{m}(M) = 0$ for $i < \dim M$.
\end{cor}

\begin{lem} \label{union of clutters 1}
Let $S=K[x_1, \ldots, x_n, y_1, \ldots, y_m]$ be the polynomial ring and $I$ be an ideal in $K[y_1, \ldots, y_m]$. Then,
$$\depth \dfrac{S}{\left( x_1\cdots x_n \right) I} = \depth \dfrac{S}{I}.$$
\end{lem}

\begin{defn}[Alexander duality] \rm
For a square-free monomial ideal $I= \left( M_1, \ldots, M_q \right) \subset K[x_1, \ldots, x_n]$, the Alexander dual $I^\vee$, of $I$ is defined to be
$$I^\vee= P_{M_1} \cap \cdots \cap P_{M_q} $$
where $P_{M_i}$ is prime ideal generated by $\{x_j \colon x_j | M_i \}$.
\end{defn}

\begin{defn} \rm
Let $I$ be a non-zero homogeneous ideal of $S$. For every $i \in \NN$ one defines
$$t^S_i(I) = \max \{j \colon \beta^S_{i,j}(I) \neq 0 \}$$
where $\beta^S_{i,j}(I)$ is the $i,j$-th graded Betti number of $I$ as an $S$-module. The \textit{Castelnuovo-Mumford regularity} of $I$, is given by
$$\reg (I)= \sup \{t^S_i(I)-i \colon i \in \ZZ \}.$$
We say that the ideal $I$ has a \textit{$d$-linear resolution} if $I$ is generated by homogeneous polynomials of degree $d$ and $\beta_{i,j}^S (I) =0$ for all $j \neq i+d$ and $i \geq 0$.
For an ideal which has $d$-linear resolution, the Castelnuovo-Mumford regularity would be $d$.
\end{defn}

\begin{thm}[{Eagon-Reiner \cite[Theorem 3]{Eagon-Reiner}}] \label{Eagon-Reiner}
Let $I$ be a square-free monomial ideal in $S=K[x_1, \ldots, x_n]$. $I$ has $q$-linear resolution if and only if $S/I^\vee$ is Cohen-Macaulay of dimension $n - q$.
\end{thm}

\begin{thm}[{\cite[Theorem 2.1]{Terai}}]\label{Terai}
Let $I$ be a square-free monomial ideal in $S=K[x_1, \ldots, x_n]$ with $\dim S/I \leq n-2$. Then,
$$\dim \frac{S}{I^\vee} - \depth \frac{S}{I^\vee}= \reg(I) - {\rm indeg} (I).$$
\end{thm}

\begin{rem} \rm \label{Passing to regularity}
Let $I, J$ be square-free monomial ideals generated by elements of degree $d\geq 2$ in $S=K[x_1, \ldots, x_n]$. By Theorem~\ref{Terai}, we have
\begin{equation*}
\begin{split}
\reg (I) = n - \depth\frac{S}{I^\vee}, \qquad \reg (J) = n - \depth\frac{S}{J^\vee}.
\end{split}
\end{equation*}
Therefore, $\reg (I) = \reg (J)$ if and only if $\depth {S}/{I^\vee} = \depth {S}/{J^\vee}$.
\end{rem}

\begin{defn}[Clutter] \rm
Let $[n]=1, \ldots, n$. A \textit{clutter} $\C$ on a vertex set $[n]$ is a set of subsets of $[n]$ (called \textit{circuits} of $\C$) such that if $e_1$ and $e_2$ are distinct circuits of $\C$ then $e_1 \nsubseteq e_2$.
A \textit{$d$-circuit} is a circuit consisting of exactly $d$ vertices, and a clutter is \textit{$d$-uniform} if every circuit has exactly $d$ vertices.
\end{defn}

For a non-empty clutter $\C$ on vertex set $[n]$, we define the ideal $I(\C)$, as follows:
$$I(\C) = \left( \textbf{x}_F \colon \quad F \in \C \right)$$
and we define $I(\varnothing) = 0$.

Let $n, d$ be positive integers and $d \leq n$. We define $\C_{n,d}$, the \textit{maximal $d$-uniform clutter on $[n]$} as follow:
$$\C_{n,d}=\{F \subset [n] \colon \quad |F|=d\}.$$
If $\C$ is a $d$-uniform clutter on $[n]$, we define $\bar{\C}$, the \textit{complement} of $\C$, to be
$$\bar{\C}= \C_{n,d} \setminus \C = \{F \subset [n] \colon \quad |F|=d, \,F \notin \C\}.$$
Frequently in this paper, we take a $d$-uniform clutter $\C$ and we consider the square-free ideal $I=I(\bar{\C})$ in the polynomial ring $S=K[x_1, \ldots, x_n]$. The ideal $I$ is called \textit{circuit ideal}.

\begin{defn}[Clique] \rm
Let $\C$ be a $d$-uniform clutter on $[n]$. A subset $G \subset [n]$ is called a \textit{clique} in $\C$, if all $d$-subset of $G$ belongs to $\C$.
\end{defn}

\begin{rem} \rm \label{clique}
Let $\C$ be a $d$-uniform clutter on $[n]$ and $I=I(\bar{\C})$ be the circuit ideal. If $G$ is a clique in $\C$ and $F \in \bar{\C}$, then $([n] \setminus G) \cap F \neq \varnothing$. So that $\textbf{x}_{[n] \setminus G} \in P_F$. Hence
$$\textbf{x}_{[n] \setminus G} \in \bigcap\limits_{F \in \bar{\C}} P_F = I^\vee.$$
\end{rem}

\begin{ex} \label{linearity of C_n,d}\rm
We show that $I(\C_{n,d})$ has linear resolution.
Let $\Delta$ be a simplex on $[n]$. Then, clearly $I_\Delta = (0)$ and $K[\Delta] = K[x_1, \ldots, x_n]$ is Cohen-Macaulay. It follows from {\cite[Exercise 5.1.23]{BH}}
that for any $r<n$, $\Delta^{(r)} = \langle F \subset [n] \colon \quad |F|=r+1 \rangle$ is Cohen-Macaulay. Note that
$$I_{\Delta^{(r)}}^\vee = I \left( \overline{\Delta^{(r)}} \right)= \left( \text{\rm \textbf{x}}_F \colon \quad |F|= n-(r+1) \right)$$
which has linear resolution by Theorem~\ref{Eagon-Reiner}.
Using this argument for $r= n-d-1$, one can say, $I_{\Delta^{(n-d-1)}}^\vee = I(\C_{n,d})$ has linear resolution.
\end{ex}

\begin{defn}[Simplicial submaximal circuit] \rm
Let $\C$ be a $d$-uniform clutter on $[n]$. A $(d-1)$-subset $e \subset [n]$ is called a \textit{submaximal circuit} of $\C$ if there exists $F \in \C$ such that $e \subset F$. The set of all submaximal circuits of $\C$ is denoted by $E(\C)$.
For $e \in E(\C)$, let ${\rm N}[e]= e \cup \big\{ c \in [n] \colon e \cup \{c\} \in \C \big\} \subset [n]$.
We say that $e$ is a \textit{simplicial submaximal circuit} if ${\rm N}[e]$ is a clique in $\C$. In case of 3-uniform clutters, $E(\C)$ is called the edge set and we say \textit{simplicial edge} instead of simplicial submaximal circuit.
\end{defn}

\section{Operations on Clutters}

In this section, for a clutter $\C$, we introduce some operations as changing or removing circuits which do not change linearity of resolution of the circuit ideal. We begin this section with the following well-known results.

\begin{lem}\label{Moduli sum rule}
Let $M$ be an $R$-module. For any submodules $A,B,C$ of $M$ such that $B \subset C$, one has
\begin{equation}
(A+B) \cap C = (A \cap C) + B .
\end{equation}
\end{lem}

\begin{thm}[Mayer-Vietoris sequence] \label{Mayer-Vietoris sequence}
For any two ideals $I_1, I_2$ in the commutative Noetherian local ring $(R, \mathfrak{m})$, the short exact sequence
$$ 0 \longrightarrow \frac{R}{I_1 \cap I_2} \longrightarrow \frac{R}{I_1} \oplus \frac{R}{I_2} \longrightarrow \frac{R}{I_1+I_2} \longrightarrow 0 $$
gives rise to the long exact sequence
\begin{equation*} \label{Mayer}
\begin{split}
\cdots \rightarrow H_\mathfrak{m}^{i-1} \left( \tfrac{R}{I_1+I_2}\right) \rightarrow H_\mathfrak{m}^i \left(\tfrac{R}{I_1 \cap I_2}\right)
 \longrightarrow H_\mathfrak{m}^i \left(\tfrac{R}{I_1}\right) \oplus H_\mathfrak{m}^i \left(\tfrac{R}{I_2}\right)
 \longrightarrow H_\mathfrak{m}^i \left(\tfrac{R}{I_1+I_2}\right) \rightarrow \\
 \rightarrow H_\mathfrak{m}^{i+1} \left( \tfrac{R}{I_1+I_2}\right) \rightarrow \cdots.
\end{split}
\end{equation*}
\end{thm}

\begin{lem} \label{sum}
Let $I_1, I_2$ be ideals in a commutative Noetherian local ring $(R, \mathfrak{m})$ such that
$$\depth \tfrac {R}{I_1} \geq \depth \tfrac{R}{I_2} > \depth \tfrac{R}{I_1+I_2}.$$
Then, $\depth \frac{R}{I_1 \cap I_2} = 1+ \depth \frac{R}{I_1+I_2}.$
\end{lem}

\begin{proof}
Let $r := 1+ \depth R/(I_1+I_2)$. Then, for all $i < r$,
$$H_\mathfrak{m}^{i-1} \left(\tfrac{R}{I_1+I_2}\right)= H_\mathfrak{m}^{i} \left(\tfrac{R}{I_1}\right) = H_\mathfrak{m}^{i} \left(\tfrac{R}{I_2}\right)=0.$$
Hence by the Mayer-Vietoris exact sequence,
\begin{equation*}
\begin{split}
\cdots \rightarrow H_\mathfrak{m}^{i-1} \left(\tfrac{R}{I_1}\right) \oplus H_\mathfrak{m}^{i-1} \left(\tfrac{R}{I_2}\right) \rightarrow  H_\mathfrak{m}^{i-1} \left(\tfrac{R}{I_1+I_2}\right) \rightarrow H_\mathfrak{m}^i \left(\tfrac{R}{I_1 \cap I_2}\right)  \rightarrow H_\mathfrak{m}^i \left(\tfrac{R}{I_1}\right) \oplus H_\mathfrak{m}^i \left(\tfrac{R}{I_2}\right) \rightarrow \cdots
\end{split}
\end{equation*}
we have $H_\mathfrak{m}^i \left(\tfrac{R}{I_1 \cap I_2}\right) =0$ for all $i < r$, and  $H_\mathfrak{m}^{r}  \left(\tfrac{R}{I_1 \cap I_2}\right) \neq 0$.
So that
$$\depth \frac{R}{I_1 \cap I_2} = r = 1+ \depth \frac {R}{I_1+I_2}.$$
\end{proof}

\begin{lem} \label{Marcel}
Let $I, I_1, I_2$ be ideals in a commutative Noetherian local ring $(R, \mathfrak{m})$ such that $I=I_1+I_2$ and
$$r:= \depth \frac{R}{I_1 \cap I_2} \leq \depth \frac{R}{I_2}.$$
Then, for all $i<r-1$ one has
\begin{equation*}
H_\mathfrak{m}^i \left( \frac{R}{I_1} \right) \cong H_\mathfrak{m}^i \left( \frac{R}{I} \right).
\end{equation*}
\end{lem}

\begin{proof}

For $i<r-1$, our assumption implies that
\begin{equation*}
H_\mathfrak{m}^i \left(\tfrac{R}{I_1 \cap I_2}\right) = H_\mathfrak{m}^i \left(\tfrac{R}{I_2}\right) = H_\mathfrak{m}^{i+1} \left(\tfrac{R}{I_1 \cap I_2}\right) = 0.
\end{equation*}
Hence, from the Mayer-Vietoris exact sequence
\begin{equation*}
\begin{split}
\cdots \longrightarrow H_\mathfrak{m}^i \left(\tfrac{R}{I_1 \cap I_2}\right) \longrightarrow H_\mathfrak{m}^i \left(\tfrac{R}{I_1}\right) \oplus H_\mathfrak{m}^i \left(\tfrac{R}{I_2}\right) \longrightarrow H_\mathfrak{m}^i \left(\tfrac{R}{I}\right)
\longrightarrow H_\mathfrak{m}^{i+1} \left(\tfrac{R}{I_1 \cap I_2}\right) \longrightarrow \cdots.
\end{split}
\end{equation*}
we have
\begin{equation*}
H_\mathfrak{m}^i \left( \frac{R}{I_1} \right) \cong H_\mathfrak{m}^i \left( \frac{R}{I} \right), \qquad \text{for all } i<r-1,
\end{equation*}
as desired.
\end{proof}

\noindent \textbf{Notation.} Let for $n>3$, $T_{1,n}, T'_{1,n} \subset S= K[x_1, \ldots, x_n]$ denote the ideals
\begin{equation*}
T_{1,n} = \bigcap\limits_{2 \leq i < j \leq n} (x_1, x_i, x_j), \ \ \ \
T'_{1,n}= \bigcap\limits_{2 \leq i < j \leq n} (x_i, x_j).
\end{equation*}

\begin{prop}\label{T1n}
For $n \geq 3$, let $S= K[x_1, \ldots, x_n]$ be the polynomial ring, then
\begin{itemize}
\item[\rm (i)] $T'_{1,n}= \bigg( \prod\limits_{\substack{2 \leq i  \leq n \atop i \neq 2}}x_i, \ldots, \prod\limits_{\substack{2 \leq i \leq n \atop i \neq n}}x_i \bigg)$ \ \ and \ \  $T_{1,n} = \bigg( x_1, \prod\limits_{\substack{2 \leq i  \leq n \atop i \neq 2}}x_i, \ldots, \prod\limits_{\substack{2 \leq i \leq n \atop i \neq n}}x_i \bigg)$.
\item[\rm (ii)] $\dfrac{S}{T'_{1,n}}$  (res. $\dfrac{S}{T_{1,n}}$) is Cohen-Macaulay of dimension $n-2$ (res. $n-3$).
\end{itemize}
\end{prop}

\begin{proof}
The assertion is well-known but one can find the a direct proof of the primary decomposition of the Alexander dual of $T'_{1,n}$ in {\cite[Example 7]{Morales}}.
\end{proof}

Let $\C$ be a $3$-uniform clutter on the vertex set $[n]$. Surely, one can consider $\C$ as a $3$-uniform clutter on $[m]$ for any $m \geq n$. However, $\bar{\C}$ (and hence $I(\bar{\C})$) will be changed when we consider $\C$ either on $[n]$ or on $[m]$. To be more precise, when we pass from $[n]$ to $[n+1]$, then the new generators $\{x_{n+1}x_ix_j \colon 1 \leq i<j \leq n \}$ will be added to $I(\bar{\C})$. Below, we will show that the linearity does not change when we pass from $[n]$ to $[m]$.

\begin{lem} \label{redundant vertex}
Let $I \subset K[x_1, \ldots, x_n]$ be a square-free monomial ideal generated in degree $3$ such that $x_1x_ix_j \in I$ for all $1 < i < j \leq n$. If $J= I \cap K[x_2, \ldots, x_n]$, then $\reg (I) = \reg (J)$.
\end{lem}

\begin{proof}
By our assumption, $J$ is an ideal of $K[x_2, \ldots, x_n]$ and
$$I = J + (x_1x_ix_j \colon 1 < i < j \leq n).$$
So that $I^\vee = J^\vee \bigcap T_{1,n}$. By Remark~\ref{Passing to regularity}, it is enough to show that $\depth S/I^\vee = \depth S/J^\vee$.

The ideal $J^\vee$ is intersection of some primes $P$, such that the set of generators of $P$ is a $3$-subset of $\{x_1, \ldots, x_n \}$.
So that for all $j$, $\prod\limits_{\substack{1<i  \leq n-1 \atop i \neq j}}x_i \in J^\vee$. Hence $J^\vee + T_{1,n} = (x_1, J^\vee)$ by Proposition~\ref{T1n}(i). In particular
\begin{equation} \label{a}
\depth \frac{S}{J^\vee + T_{1,n}} = \depth \frac{S}{J^\vee}- 1.
\end{equation}
By Proposition~\ref{T1n} and (\ref{a}), $\depth \frac{S}{T_{1,n}} \geq \depth \frac{S}{J^\vee} > \depth \frac{S}{J^\vee + T_{1,n}}$. Hence by Lemma~\ref{sum} and (\ref{a}), we have
$$\depth \frac{S}{I^\vee} = 1+\depth \frac{S}{J^\vee + T_{1,n}} = \depth \frac{S}{J^\vee}.$$
\end{proof}

\begin{thm}\label{simplicial edge}
Let $\C\neq \C_{n,d}$ be a $d$-uniform clutter on $[n]$ and $e$ be a simplicial submaximal circuit. Let
\begin{equation*}
\C' = \C \setminus e = \big\{ F \in \C \colon e \nsubseteq F \big\}
\end{equation*}
and $I=I(\bar{\C}), J=I(\bar{\C}')$. Then, $\reg (I) = \reg (J)$.
\end{thm}

\begin{proof}
By Remark~\ref{Passing to regularity}, it is enough to show that $\depth S/I^\vee = \depth S/J^\vee$. Without loss of generality, we may assume that $e =\{1, \ldots, d-1\}$ and $N[e]=\{1, \ldots, r\}$.

Since $e=\{1,\ldots, d-1\}$ is a simplicial submaximal circuit, by Remark~\ref{clique} and Lemma~\ref{Moduli sum rule}, we have:
\begin{align*}
I^\vee & = (x_1, \ldots, x_{d-1}, x_{r+1} \cdots x_n) \cap \bigg( \bigcap\limits_{F \in \bar{\C} \atop \{1,\ldots, d-1\} \nsubseteq F} P_F \bigg) \\
& = \bigg[ (x_1, \ldots, x_{d-1}) \cap \big( \bigcap\limits_{F \in \bar{\C} \atop \{1,\ldots, d-1\} \nsubseteq F} P_F \big) \bigg] + (x_{r+1} \cdots x_n),\\
J^\vee  & = (x_1, \ldots, x_{d-1}, x_{d} \cdots x_n) \cap \bigg( \bigcap\limits_{F \in \bar{\C} \atop \{1,\ldots, d-1\} \nsubseteq F} P_F \bigg) \\
& = \bigg[ (x_1, \ldots, x_{d-1}) \cap \big( \bigcap\limits_{F \in \bar{\C} \atop \{1,\ldots, d-1\} \nsubseteq F} P_F \big) \bigg] + (x_{d} \cdots x_n).
\end{align*}
Since
\begin{align*}
& (x_1, \ldots, x_{d-1}) \cap \bigg( \bigcap\limits_{F \in \bar{\C} \atop \{1,\ldots, d-1\} \nsubseteq F} P_F \bigg) \cap (x_{r+1} \cdots x_n) = (x_1x_{r+1} \cdots x_n, \ldots, x_{d-1}x_{r+1} \cdots x_n),\\
& (x_1, \ldots, x_{d-1}) \cap \bigg( \bigcap\limits_{F \in \bar{\C} \atop \{1,\ldots, d-1\} \nsubseteq F} P_F \bigg) \cap (x_{d} \cdots x_n) = (x_1x_{d} \cdots x_n, \ldots, x_{d-1}x_{d} \cdots x_n)
\end{align*}
have depth equal to $n-(d-1)$, by Lemma~\ref{Marcel} we have:
\begin{align} \label{Eq1}
H_\mathfrak{m}^i \left( \frac{S}{I^\vee} \right) \cong H_\mathfrak{m}^i \bigg( \frac{S}{(x_1, \ldots, x_{d-1}) \cap \big( \bigcap\limits_{F \in \bar{\C} \atop \{1,\ldots, d-1\} \nsubseteq F} P_F \big)} \bigg) \cong H_\mathfrak{m}^i \left( \frac{S}{J^\vee} \right)\quad \text{for all }  i<n-d.
\end{align}
Since $\dim S/I^\vee = \dim S/J^\vee = n-d$, the above equation implies that $\depth S/I^\vee = \depth S/J^\vee$.
\end{proof}

For a $d$-uniform clutter $\C$, if there exist only one circuit $F \in \C$ which contains the submaximal circuit $e \in E(\C)$, then clearly $e$ is a simplicial submaximal circuit. Hence we have the following result.

\begin{cor}\label{one vertex}
Let $\C$ be a $d$-uniform clutter on $[n]$ and $I=I(\bar{\C})$ be the circuit ideal. If $F$ is the only circuit containing the submaximal circuit $e$, then $\reg (I)= \reg (I+ {\text {\rm \textbf{x}}}_F)$.
\end{cor}

Let $\C$ be $3$-uniform clutter on $[n]$ such that $\{1,2,3\}, \{1,2,4\}, \{1,3,4\}, \{2,3,4\} \in \C$. If there exist no other circuit which contains $e= \{1,2\}$, then $e$ is a simplicial edge. Hence by Theorem~\ref{simplicial edge} we have the following corollary.

\begin{thm} \label{x5theorem}
Let $\C$ be $3$-uniform clutter on $[n]$ and $I=I(\bar{\C})$ be the circuit ideal of $\bar{\C}$. Assume that $\{1,2,3\}, \{1,2,4\}, \{1,3,4\}, \{2,3,4\} \in \C$
and there exist no other circuit which contains $\{1,2\}$. If $J=I+(x_1x_2x_3, x_1x_2x_4)$, then $\reg (I) = \reg(J)$.
\end{thm}

E. Emtander in \cite{Emtander} has introduced  a generalized chordal clutter to be a $d$-uniform clutter, obtained inductively as follows:
\begin{itemize}
\item[$\bullet$] $\C_{n,d}$ is a generalized chordal clutter.
\item[$\bullet$] If $\mathcal{G}$ is generalized chordal clutter, then so is $\C = \mathcal{G} \cup_{\C_{i,d}} {\C_{n,d}}$ for all $0 \leq i <n$.
\item[$\bullet$] If $\mathcal{G}$ is generalized chordal and $V \subset V(\mathcal{G})$ is a finite set with $|V| = d$ and at least one element of $\{F \subset V: |F|=d-1\}$ is not a subset of any element of $\mathcal{G}$, then $\mathcal{G} \cup V$ is generalized chordal.
\end{itemize}

Also R. Woodroofe in \cite{Wood1} has defined a simplicial vertex in a $d$-uniform clutter to be a vertex $v$ such the if it belongs to two circuits $e_1, e_2$, then, there is another circuit in $(e_1\cup e_2)\setminus\{v\}$. He calls a clutter chordal if any minor of the clutter has a simplicial vertex.

\begin{rem} \rm \label{reducable to zero}
Let $\mathscr{C}$ be the class of $3$-uniforms clutters which can be transformed to empty set after a sequence of deletions of simplicial edges. Using Theorem~\ref{simplicial edge}, it is clear that if $\C \in \mathscr{C}$, then the ideal $I(\bar{\C})$ has a linear resolution over any field $K$. It is easy to see that generalized $3$-uniform chordal clutters are contained in this class, so they have linear resolution over any field $K$. This generalizes Theorem 5.1 of \cite{Emtander}. It is worth to say that $\mathscr{C}$ contains generalized chordal clutters strictly. For example, $\C=\{123, 124, 134, 234, 125, 126, 156, 256\}$ is in $\mathscr{C}$ but it is not a generalized chordal clutter. Also it is easy to see that any 3-uniform clutter which is chordal in sense of \cite{Wood1} has simplicial edges.
\end{rem}

\begin{defn}[Flip] \rm
Let $\C$ be $3$-uniform clutter on $[n]$. Assume that $\{ 1, 2, 3 \}, \{ 1, 2, 4 \} \in \C$ are the only circuits containing $\{ 1, 2 \}$ and there is no circuit in $\C$ containing $\{ 3, 4 \}$. Let $\C'= \C \cup \big\{ \{ 1, 3, 4 \}, \{ 2, 3, 4 \} \big\} \setminus \big\{ \{ 1, 2, 3 \}, \{ 1, 2, 4 \} \big\}$. Then $\C'$ is called a \textit{flip} of $\C$.
Clearly, if $\C'$ is a flip of $\C$, then  $\C$ is a flip of $\C'$ too (see the following illustration).
\end{defn}
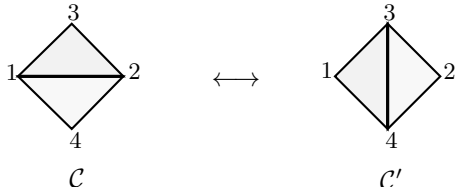
\begin{center}
\psset{xunit=0.7cm,yunit=0.7cm,algebraic=true,dotstyle=o,dotsize=3pt 0,linewidth=0.8pt,arrowsize=3pt 2,arrowinset=0.25}
\begin{pspicture*}(-0.32,-0.1)(9.94,3.8)
\pspolygon[fillstyle=solid,fillcolor=gray,opacity=0.1](2,3)(0.98,2)(2.98,2)
\pspolygon[fillstyle=solid,fillcolor=lightgray,opacity=0.1](0.98,2)(2,1)(2.98,2)
\pspolygon[fillstyle=solid,fillcolor=gray,opacity=0.1](7,2)(8,3)(8,1)
\pspolygon[fillstyle=solid,fillcolor=lightgray,opacity=0.1](8,3)(9,2)(8,1)
\psline[linewidth=1.2pt](0.98,2)(2.98,2)
\psline[linewidth=1.2pt](8,3)(8,1)
\rput[tl](0.75,2.22){\small 1}
\rput[tl](3.08,2.24){\small 2}
\rput[tl](1.92,3.36){\small 3}
\rput[tl](1.96,0.94){\small 4}
\rput[tl](6.73,2.26){\small 1}
\rput[tl](9.06,2.24){\small 2}
\rput[tl](7.92,3.36){\small 3}
\rput[tl](7.96,0.94){\small 4}
\rput[tl](4.66,2.0){$\longleftrightarrow$}
\rput[tl](1.96,0.24){$\C$}
\rput[tl](7.86,0.24){$\C'$}
\end{pspicture*}
\end{center}

\begin{cor} \label{Flip}
Let $\C$ be $3$-uniform clutter on $[n]$ and $\C'$ be a flip of $\C$. Then, $\reg I(\bar{\C}) = \reg I(\bar{\C}')$.
\end{cor}

\begin{proof}
With the same notation as in the above definition, let $\C''= \C \cup \big\{ \{ 1, 3, 4 \}, \{ 2, 3, 4 \} \big\}$. Theorem~\ref{x5theorem} applied to $\{3,4\}$, shows that $\reg I(\bar{\C}'') = \reg I(\bar{\C}')$. Using Theorem~\ref{x5theorem} again applied to $\{1,2\}$, we conclude that $\reg I(\bar{\C}'')= \reg I(\bar{\C})$. So that $\reg I(\bar{\C})= \reg I(\bar{\C}')$, as desired.
\end{proof}

For our next theorem, we use the following lemmas.

\begin{lem}\label{lemma for karbordi}
Let $n \geq 4, S=K[x_1, \ldots, x_n]$ be the polynomial ring and $T_n$ be the ideal
$$T_n = (x_4 \cdots x_n, x_1x_2x_3\;\hat{x}_4 \cdots x_n, \ldots, x_1x_2x_3\;{x}_4 \cdots \hat{x}_n).$$
Then, we have:
\begin{itemize}
\item[\rm (i)] $T_n = \left( T_{n-1} \cap (x_n) \right) +  (x_1x_2x_3\;{x}_4 \cdots \hat{x}_n)$.
\item[\rm (ii)] $\depth \dfrac{S}{T_n} = n-2$.
\end{itemize}
\end{lem}

\begin{proof}
(i) This is an easy computation.

(ii) The proof is on induction over $n$. For $n=4$, every thing is clear. Let $n>4$ and (ii) be true for $n-1$.\\
Clearly, $\left( T_{n-1} \cap (x_n) \right) \cap  (x_1x_2x_3\;{x}_4 \cdots \hat{x}_n) = (x_1x_2x_3\;{x}_4 \cdots {x}_n)$ which has depth $n-1$. So by Lemma~\ref{Marcel}, \ref{union of clutters 1} and induction hypothesis, we have:
$$\depth \dfrac{S}{T_n} = \depth \dfrac{S}{T_{n-1}} = n-2.$$
\end{proof}

\begin{lem}\label{lemma for karbordi 2}
Let $\C$ be a $3$-uniform clutter on $[n]$ such that $F=\{1, 2, 3\} \in \C$ and for all $r > 3$,
\begin{equation} \label{Local Property 1}
\big\{ \{1,2,r\}, \{1,3,r\}, \{2,3,r\} \big\}  \nsubseteq \C.
\end{equation}
Let $\C_1 = \C \setminus F$ and $I=I(\bar{\C}), I_1=I(\bar{\C}_1)$. Then,
\begin{itemize}
\item[\rm (i)] $\depth \dfrac{S}{I^\vee+(x_1, x_2,x_3)} \geq \depth \dfrac{S}{I^\vee} - 1$.
\item[\rm (ii)] $\depth \dfrac{S}{I_1^\vee} \geq \depth \dfrac{S}{I^\vee}$.
\end{itemize}
\end{lem}

\begin{proof}
Let $t:= \depth {S}/{I^\vee} \leq \dim {S}/{I^\vee} = n-3$.\\
(i) One can easily check that condition (\ref{Local Property 1}), is equivalent to say that
\begin{center}
for all $r>3$, there exist $F \in \bar{\C}$ such that $P_F \subset (x_1, x_2, x_3, x_r)$.
\end{center}
So that
\begin{equation*}
\begin{split}
I^\vee =  \bigcap\limits_{F \in \bar{\C}} P_F & = \left( \bigcap\limits_{F \in \bar{\C}} P_F \right) \cap \left( (x_1, x_2, x_3, x_4) \cap \cdots \cap (x_1, x_2, x_3, x_n) \right) \\
& = \left( \bigcap\limits_{F \in \bar{\C}} P_F \right) \cap (x_1, x_2, x_3, x_4 \cdots x_n) = I^\vee \cap (x_1, x_2, x_3, x_4 \cdots x_n).
\end{split}
\end{equation*}
Clearly, $x_4 \cdots x_n \in I^\vee$. So, from the Mayer-Vietoris long exact sequence
\begin{equation*}
\begin{split}
\cdots \rightarrow H_\mathfrak{m}^{i-1} \left( \tfrac{S}{I^\vee} \right) \oplus H_\mathfrak{m}^{i-1} \left( \tfrac{S}{(x_1, x_2, x_3, x_4 \cdots x_n)} \right) \rightarrow H_\mathfrak{m}^{i-1} \left( \tfrac{S}{I^\vee + (x_1, x_2, x_3)} \right)
\rightarrow  H_\mathfrak{m}^{i} \left( \tfrac{S}{I^\vee} \right) \rightarrow \cdots
\end{split}
\end{equation*}
we have:
\begin{equation} \label{Local EQ 7}
H_\mathfrak{m}^{i-1} \left( \frac{S}{I^\vee + (x_1, x_2, x_3)} \right)=0, \qquad \text{for all } i<t \leq n-3.
\end{equation}
This proves inequality (i).\\
(ii) Clearly, $I_1^\vee = I^\vee \cap (x_1, x_2, x_3)$. So from Mayer-Vietoris long exact sequence
\begin{equation*}
\begin{split}
\cdots \rightarrow H_\mathfrak{m}^{i-1} \left( \tfrac{S}{I^\vee + (x_1, x_2, x_3)} \right) \rightarrow  H_\mathfrak{m}^{i} \left( \tfrac{S}{I_1^\vee} \right) \rightarrow H_\mathfrak{m}^{i} \left( \tfrac{S}{I^\vee} \right) \oplus H_\mathfrak{m}^{i} \left( \tfrac{S}{(x_1, x_2, x_3)} \right) \rightarrow \cdots
\end{split}
\end{equation*}
and (\ref{Local EQ 7}), we have:
\begin{equation*}
H_\mathfrak{m}^{i} \left( \frac{S}{I_1^\vee} \right)=0, \qquad \text{for all } i < t \leq n-3.
\end{equation*}
\end{proof}

\begin{thm}\label{karbordi}
Let $\C$ be a $3$-uniform clutter on $[n]$ such that $F=\{1, 2, 3\} \in \C$ and for all $r > 3$, $\big\{ \{1,2,r\}, \{1,3,r\}, \{2,3,r\} \big\}  \nsubseteq \C$. Let $\C_1 = \C \setminus F$, $\C'= \C_1 \cup \big\{ \{0,1,2\}, \{0,1,3\}, \{0,2,3\} \big\}$ and $I=I(\bar{\C}), J=I(\bar{\C}')$ be the circuit ideals in the polynomial ring $S=K[x_0, x_1, \ldots, x_n]$. Then, $\reg (I) = \reg (J)$.
\end{thm}

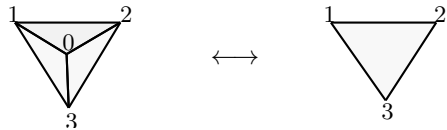
\begin{center}
\psset{xunit=0.7cm,yunit=0.7cm,algebraic=true,dotstyle=o,dotsize=3pt 0,linewidth=0.8pt,arrowsize=3pt 2,arrowinset=0.25}
\begin{pspicture*}(-0.4,1)(12.14,3.4)
\pspolygon[linestyle=none,fillstyle=solid,fillcolor=lightgray,opacity=0.1](8,3)(10,3)(9.04,1.52)
\pspolygon[linestyle=none,fillstyle=solid,fillcolor=lightgray,opacity=0.1](4,3)(2.98,2.4)(3.02,1.38)
\pspolygon[linestyle=none,fillstyle=solid,fillcolor=lightgray,opacity=0.1](2,3)(2.98,2.4)(3.02,1.38)
\pspolygon[linestyle=none,fillstyle=solid,fillcolor=gray,opacity=0.1](2,3)(4,3)(2.98,2.4)
\psline(8,3)(10,3)
\psline(10,3)(9.04,1.52)
\psline(9.04,1.52)(8,3)
\psline(4,3)(2.98,2.4)
\psline(2.98,2.4)(3.02,1.38)
\psline(3.02,1.38)(4,3)
\psline(2,3)(2.98,2.4)
\psline(2.98,2.4)(3.02,1.38)
\psline(3.02,1.38)(2,3)
\psline(2,3)(4,3)
\psline(4,3)(2.98,2.4)
\psline(2.98,2.4)(2,3)
\rput[tl](5.72,2.4){$\longleftrightarrow$}
\rput[tl](1.86,3.3){\small 1}
\rput[tl](4.0,3.3){\small 2}
\rput[tl](2.96,1.26){\small 3}
\rput[tl](2.9,2.76){\small 0}
\rput[tl](7.86,3.3){\small 1}
\rput[tl](9.96,3.3){\small 2}
\rput[tl](8.96,1.46){\small 3}
\end{pspicture*}
\end{center}

\begin{proof}
By Remark~\ref{Passing to regularity}, it is enough to show that $\depth S/I^\vee = \depth S/J^\vee$.

Let $I_1=I(\bar{\C}_1)$. Clearly, $I_1^\vee = (x_1, x_2, x_3) \cap I^\vee$ and
\begin{equation*}
\begin{split}
J^\vee &= I_1^\vee \cap \left( \bigcap\limits_{i=4}^{n} (x_0, x_1, x_i) \right) \cap \left( \bigcap\limits_{i=4}^{n} (x_0, x_2, x_i) \right) \cap \left( \bigcap\limits_{3 \leq i < j \leq n} (x_0, x_i, x_j) \right) \\
& = (x_0, x_4 \cdots x_n, x_1x_2x_3\;\hat{x}_4 \cdots x_n, \ldots, x_1x_2x_3\;{x}_4 \cdots \hat{x}_n ) \cap I_1^\vee.
\end{split}
\end{equation*}
Let $T$ be the ideal $T= (x_0, x_4 \cdots x_n, x_1x_2x_3\;\hat{x}_4 \cdots x_n, \ldots, x_1x_2x_3\;{x}_4 \cdots \hat{x}_n ).$
Then, $J^\vee = I_1^\vee \cap T$ and by Lemma~\ref{lemma for karbordi}, $\depth \frac{S}{T} = n-2$.
Moreover, our assumption implies that for all $i>4$, there exist $F \in \bar{\C}$ such that $P_F \subset (x_1, x_2, x_3, x_r)$. So that
\begin{align} \label{aa}
I_1^\vee + T & = (x_0, x_4 \cdots x_n , I_1^\vee)  = (x_0) + \left( x_4 \cdots x_n , \left[ (x_1, x_2, x_3) \cap \left( \bigcap\limits_{F \in \bar{\C}} P_F \right) \right] \right) \nonumber \\
& = (x_0) + \left( (x_1, x_2, x_3, x_4) \cap \cdots \cap (x_1, x_2, x_3, x_n) \cap \left( \bigcap\limits_{F \in \bar{\C}} P_F \right) \right) = (x_0) + \left( \bigcap\limits_{F \in \bar{\C}} P_F \right) = (x_0, I^\vee).
\end{align}
Hence, by Lemma~\ref{lemma for karbordi 2}(ii), $\depth \frac{S}{I_1^\vee + T} = \depth \frac{S}{I^\vee} -1 \leq  \depth \frac{S}{I_1^\vee} -1.$
Thus, $\depth \frac{S}{T} \geq \depth \frac{S}{I_1^\vee} > \depth \frac{S}{I_1^\vee + T}$. Using Lemma~\ref{sum} and (\ref{aa}), $\depth \frac{S}{J^\vee} = 1+ \depth \frac{S}{I_1^\vee + T} = \depth \frac{S}{I^\vee}.$
\end{proof}

\begin{lem}\label{tetrahedron}
Let $\mathfrak T$ be a hexahedron. Then, the circuit ideal of  $\overline{\mathfrak T}$ does not have linear resolution. If ${\mathfrak T}'$ be the hexahedron without one circuit, then the circuit ideal of  $\overline{{\mathfrak T}'}$ has a linear resolution.
\end{lem}
\vspace*{0.1cm}
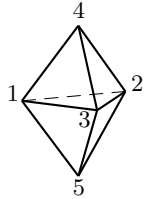
\begin{center}
 \psset{unit=0.7cm}
 \psset{unit=0.5cm}
  \begin{pspicture}(0,0)(4,4)
     \psline(-0.5,2)(1,4)
     \psline(-0.5,2)(1,0)
     \psline(-0.5,2)(1.5,1.75)
     \psline(1.5,1.75)(2.25,2.25)
     \psline(1.5,1.75)(1,4)
     \psline(2.25,2.25)(1,4)
     \psline(1.5,1.75)(1,0)
     \psline(2.25,2.25)(1,0)
     \psline[linewidth=0.3pt,linestyle=dashed](-0.5,2)(2.25,2.25)
     \rput[bl](-0.9,2){\small 1}
     \rput[bl](2.4,2.25){\small 2}
     \rput[bl](1.0,1.3){\small 3}
     \rput[bl](0.85,4.2){\small 4}
     \rput[bl](0.85,-0.5){\small 5}

     \end{pspicture}
\end{center}

\begin{proof}
Let $I= I(\bar{\mathfrak{T}})$. We know that $\bar{\mathfrak{T}} = \{ 145, 245, 345, 123 \}$. So that
$$I^\vee = (x_1x_2x_3, x_4,x_5) \cap (x_1, x_2, x_3) \subset S:= K[x_1, \ldots, x_5].$$
It follows from Theorem~\ref{Mayer-Vietoris sequence} that $H_\mathfrak{m}^{1} \left( \frac{S}{I^\vee} \right) \neq 0$. Since $\dim S/I^\vee = 5-3 =2$, we conclude that $S/I^\vee$ is not Cohen-Macaulay. So that the ideal $I$ does not have linear resolution by Theorem~\ref{Eagon-Reiner}.\\
The second part of the theorem, is a direct conclusion of Theorem~\ref{one vertex}.
\end{proof}

Let ${\mathcal S}^2$ be a sphere in $\mathbb R^3$. A \textit{triangulation} of ${\mathcal S}^2$ is a finite simple graph embedded on $\mathcal{S}^2$ such that each face is triangular and any two faces share at most one edge. Note that if $\C$ is a triangulation of a surface, then $\C$ defines a $3$-uniform clutter which we denote this again by $\C$. Moreover, any proper subclutter $\C' \subset \C$ has an edge $e \in E(\C')$ such that $e$ is contained in only one circuit of $\C'$.

\begin{cor}\label{polytope}
Let $S=K[x_1,\ldots,x_n]$. Let  ${\mathfrak P}_n$ be the clutter defined by a triangulation of the sphere with  $n \geq 5$ vertices, and let $I\subset S$ be the circuit ideal of  $\overline{\mathfrak P}_n$. Then,
\begin{itemize}
\item[\rm (i)] For any proper subclutter  $\C_1 \subset {\mathfrak P}_n$, the ideal $I(\bar{\C}_1)$ has a linear resolution.
\item[\rm (ii)] $S/I$ does not have linear resolution.
\end{itemize}
\end{cor}

\begin{proof}
(i) If $\C_1$ is a proper subclutter of ${\mathfrak P}_n$, then $\C_1$ has an edge $e$ such that $e$ is contained in only one circuit of $\C_1$ and can be deleted without changing the regularity by Corollary \ref{one vertex}. Continuing this process proves the assertion.

(ii) The proof is by induction on $n$, the number of vertices. First step of induction is Lemma~\ref{tetrahedron}. Let  $n>5$. If there is a vertex of degree 3 (the number of edges passing through the vertex in 3), then by Theorem~\ref{karbordi}, we can remove the vertex and three circuits containing it and add a new circuit instead. Then, we have a clutter with fewer vertices and by the induction hypothesis, $S/I$ does not have linear resolution. Now, assume that, there is no any vertex of degree 3 and take a vertex $u$ of degree $>3$ and all circuits containing $u$ (see the following illustrations). Using several flips and Corollary~\ref{Flip}, we can reduce our triangulation to another one such that there are only $3$ circuits containing $u$. Now, using Theorem~\ref{karbordi}, we get a triangulation of the sphere with $n-1$ vertices which does not have linear resolution by the induction hypothesis.
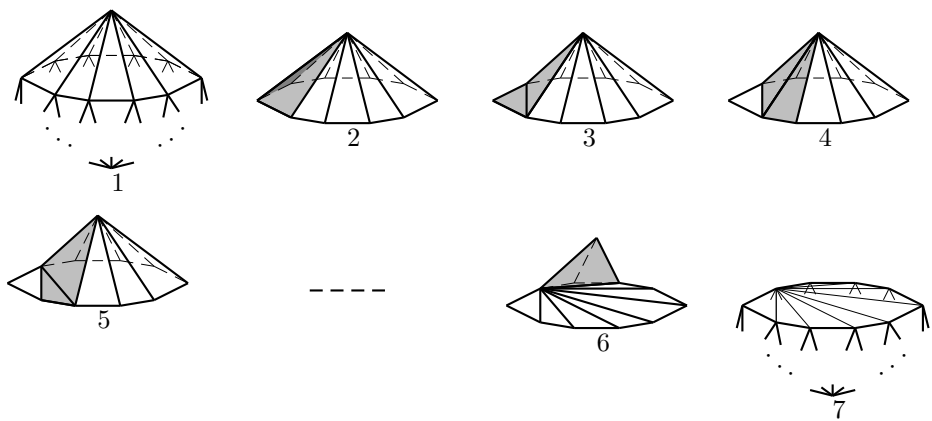
\begin{center}
 \psset{unit=0.3cm}
  \begin{pspicture}(-1,-1)(8,4)

     \psline(0,1)(1.5,0.25)
     \psline(1.5,0.25)(3,0)
     \psline(3,0)(5,0)
     \psline(5,0)(6.5,0.25)
     \psline(6.5,0.25)(8,1)
     \psline[linewidth=0.3pt,linestyle=dashed](8,1)(6.5,1.75)
     \psline[linewidth=0.3pt,linestyle=dashed](6.5,1.75)(5,2)
     \psline[linewidth=0.3pt,linestyle=dashed](5,2)(3,2)
     \psline[linewidth=0.3pt,linestyle=dashed](3,2)(1.5,1.75)
     \psline[linewidth=0.3pt,linestyle=dashed](1.5,1.75)(0,1)

     \psline(0,1)(4,4)
     \psline(1.5,0.25)(4,4)
     \psline(3,0)(4,4)
     \psline(5,0)(4,4)
     \psline(6.5,0.25)(4,4)
     \psline(8,1)(4,4)
     \psline[linewidth=0.3pt,linestyle=dashed](6.5,1.75)(4,4)
     \psline[linewidth=0.3pt,linestyle=dashed](5,2)(4,4)
     \psline[linewidth=0.3pt,linestyle=dashed](3,2)(4,4)
     \psline[linewidth=0.3pt,linestyle=dashed](1.5,1.75)(4,4)

     \psline(0,1)(-0.25,0)
     \psline(0,1)(0,-0.15)
     \psline(1.5,0.25)(1,-0.5)
     \psline(1.5,0.25)(1.7,-0.75)
     \psline(3,0)(2.6,-1)
     \psline(3,0)(3.3,-1)
     \psline(5,0)(4.6,-1)
     \psline(5,0)(5.3,-1)
     \psline(6.5,0.25)(6.3,-0.75)
     \psline(6.5,0.25)(7,-0.5)
     \psline(8,1)(8.25,0)
     \psline(8,1)(8,-0.15)
     \psline[linewidth=0.3pt,linestyle=dashed](8,1)(6.5,1.75)
     \psline[linewidth=0.3pt,linestyle=dashed](6.5,1.75)(5,2)
     \psline[linewidth=0.3pt,linestyle=dashed](5,2)(3,2)
     \psline[linewidth=0.3pt,linestyle=dashed](3,2)(1.5,1.75)
     \psline[linewidth=0.3pt,linestyle=dashed](1.5,1.75)(0,1)

     \psline[linewidth=0.3pt,linestyle=dashed](6.5,1.75)(6.25,1.25)
     \psline[linewidth=0.3pt,linestyle=dashed](5,2)(5.25,1.5)
     \psline[linewidth=0.3pt,linestyle=dashed](3,2)(3.25,1.5)
     \psline[linewidth=0.3pt,linestyle=dashed](1.5,1.75)(1.25,1.25)
     \psline[linewidth=0.3pt,linestyle=dashed](6.5,1.75)(6.75,1.25)
     \psline[linewidth=0.3pt,linestyle=dashed](5,2)(4.75,1.5)
     \psline[linewidth=0.3pt,linestyle=dashed](3,2)(2.75,1.5)
     \psline[linewidth=0.3pt,linestyle=dashed](1.5,1.75)(1.75,1.25)

     \psline(4,-3)(3,-2.75)
     \psline(4,-3)(3.5,-2.63)
     \psline(4,-3)(4,-2.5)
     \psline(4,-3)(4.5,-2.63)
     \psline(4,-3)(5,-2.75)

     \put(1,-1.5){\mbox{$\cdot$}}
     \put(7,-1.5){\mbox{$\cdot$}}
     \put(1.5,-2){\mbox{$\cdot$}}
     \put(6.5,-2){\mbox{$\cdot$}}
     \put(2,-2.3){\mbox{$\cdot$}}
     \put(6,-2.3){\mbox{$\cdot$}}

     \put(4,-4){\mbox{1}} 
     \end{pspicture}
\hspace*{0.5cm}
 \psset{unit=0.3cm}
  \begin{pspicture}(0,0)(8,4)

\pscustom{%
     \psline(0,1)(1.5,0.25)
     \gsave
     \psline(1.5,0.25)(4,4)
     \psline(0,1)(4,4)

     \fill[fillstyle=solid,fillcolor=lightgray]
\grestore}

\pscustom{%
     \psline(0,1)(4,4)
     \gsave
     \psline(0,1)(1.5,1.75)
     \psline(1.5,1.75)(4,4)
     \fill[fillstyle=solid,fillcolor=gray]
\grestore}

     \psline(0,1)(1.5,0.25)
     \psline(1.5,0.25)(3,0)
     \psline(3,0)(5,0)
     \psline(5,0)(6.5,0.25)
     \psline(6.5,0.25)(8,1)
     \psline[linewidth=0.3pt,linestyle=dashed](8,1)(6.5,1.75)
     \psline[linewidth=0.3pt,linestyle=dashed](6.5,1.75)(5,2)
     \psline[linewidth=0.3pt,linestyle=dashed](5,2)(3,2)
     \psline[linewidth=0.3pt,linestyle=dashed](3,2)(1.5,1.75)
     \psline[linewidth=0.3pt,linestyle=dashed](1.5,1.75)(0,1)

     \psline(0,1)(4,4)
     \psline(1.5,0.25)(4,4)
     \psline(3,0)(4,4)
     \psline(5,0)(4,4)
     \psline(6.5,0.25)(4,4)
     \psline(8,1)(4,4)
     \psline[linewidth=0.3pt,linestyle=dashed](6.5,1.75)(4,4)
     \psline[linewidth=0.3pt,linestyle=dashed](5,2)(4,4)
     \psline[linewidth=0.3pt,linestyle=dashed](3,2)(4,4)
     \psline[linewidth=0.3pt,linestyle=dashed](1.5,1.75)(4,4)

     \put(4,-1){\mbox{2}}
     \end{pspicture}
\hspace*{0.5cm}
 \psset{unit=0.3cm}
  \begin{pspicture}(0,0)(8,4)
\pscustom{%
     \psline(0,1)(1.5,0.25)
     \gsave
     \psline(1.5,1.75)(0,1)
     \psline(1.5,0.25)(1.5,1.75)
     \fill[fillstyle=solid,fillcolor=lightgray]
\grestore}
\pscustom{%
     \psline(1.5,0.25)(4,4)
     \gsave
     \psline(1.5,0.25)(1.5,1.75)
     \psline(1.5,1.75)(4,4)
     \fill[fillstyle=solid,fillcolor=lightgray]
\grestore}

     \psline(0,1)(1.5,0.25)
     \psline(1.5,0.25)(3,0)
     \psline(3,0)(5,0)
     \psline(5,0)(6.5,0.25)
     \psline(6.5,0.25)(8,1)
     \psline[linewidth=0.3pt,linestyle=dashed](8,1)(6.5,1.75)
     \psline[linewidth=0.3pt,linestyle=dashed](6.5,1.75)(5,2)
     \psline[linewidth=0.3pt,linestyle=dashed](5,2)(3,2)
     \psline[linewidth=0.3pt,linestyle=dashed](3,2)(1.5,1.75)
     \psline(1.5,1.75)(0,1)
     \psline(1.5,0.25)(4,4)
     \psline(1.5,0.25)(1.5,1.75)
     \psline(3,0)(4,4)
     \psline(5,0)(4,4)
     \psline(6.5,0.25)(4,4)
     \psline(8,1)(4,4)
     \psline[linewidth=0.3pt,linestyle=dashed](6.5,1.75)(4,4)
     \psline[linewidth=0.3pt,linestyle=dashed](5,2)(4,4)
     \psline[linewidth=0.3pt,linestyle=dashed](3,2)(4,4)
     \psline(1.5,1.75)(4,4)

     \put(4,-1){\mbox{3}}
     \end{pspicture}
\hspace*{0.5cm}
 \psset{unit=0.3cm}
  \begin{pspicture}(0,0)(8,4)
 \pscustom{%
     \psline(1.5,0.25)(4,4)
     \gsave
     \psline(1.5,0.25)(1.5,1.75)
     \psline(1.5,1.75)(4,4)
     \fill[fillstyle=solid,fillcolor=lightgray]
\grestore}
\pscustom{%
     \psline(1.5,0.25)(4,4)
     \gsave
     \psline(1.5,0.25)(3,0)
     \psline(3,0)(4,4)
     \fill[fillstyle=solid,fillcolor=lightgray]
\grestore}

     \psline(0,1)(1.5,0.25)
     \psline(1.5,0.25)(3,0)
     \psline(3,0)(5,0)
     \psline(5,0)(6.5,0.25)
     \psline(6.5,0.25)(8,1)
     \psline[linewidth=0.3pt,linestyle=dashed](8,1)(6.5,1.75)
     \psline[linewidth=0.3pt,linestyle=dashed](6.5,1.75)(5,2)
     \psline[linewidth=0.3pt,linestyle=dashed](5,2)(3,2)
     \psline[linewidth=0.3pt,linestyle=dashed](3,2)(1.5,1.75)
     \psline(1.5,1.75)(0,1)

     \psline(1.5,0.25)(4,4)
     \psline(1.5,0.25)(1.5,1.75)
     \psline(3,0)(4,4)
     \psline(5,0)(4,4)
     \psline(6.5,0.25)(4,4)
     \psline(8,1)(4,4)
     \psline[linewidth=0.3pt,linestyle=dashed](6.5,1.75)(4,4)
     \psline[linewidth=0.3pt,linestyle=dashed](5,2)(4,4)
     \psline[linewidth=0.3pt,linestyle=dashed](3,2)(4,4)
     \psline(1.5,1.75)(4,4)

     \put(4,-1){\mbox{4}}
     \end{pspicture}
 \end{center}

\vspace*{0.2cm}

\begin{center}
 \psset{unit=0.3cm}
  \begin{pspicture}(-1,-1)(8,4)
 \pscustom{%
     \psline(1.5,0.25)(3,0)
     \gsave
     \psline(3,0)(4,4)
     \psline(1.5,1.75)(4,4)
     \psline(1.5,1.75)(1.5,0.25)
     \fill[fillstyle=solid,fillcolor=lightgray]
\grestore}
     \psline(0,1)(1.5,0.25)
     \psline(1.5,0.25)(3,0)
     \psline(3,0)(5,0)
     \psline(5,0)(6.5,0.25)
     \psline(6.5,0.25)(8,1)
     \psline[linewidth=0.3pt,linestyle=dashed](8,1)(6.5,1.75)
     \psline[linewidth=0.3pt,linestyle=dashed](6.5,1.75)(5,2)
     \psline[linewidth=0.3pt,linestyle=dashed](5,2)(3,2)
     \psline[linewidth=0.3pt,linestyle=dashed](3,2)(1.5,1.75)
     \psline(1.5,1.75)(0,1)
     \psline(3,0)(1.5,1.75)
     \psline(1.5,0.25)(1.5,1.75)
     \psline(3,0)(4,4)
     \psline(5,0)(4,4)
     \psline(6.5,0.25)(4,4)
     \psline(8,1)(4,4)
     \psline[linewidth=0.3pt,linestyle=dashed](6.5,1.75)(4,4)
     \psline[linewidth=0.3pt,linestyle=dashed](5,2)(4,4)
     \psline[linewidth=0.3pt,linestyle=dashed](3,2)(4,4)
     \psline(1.5,1.75)(4,4)

     \put(4,-1){\mbox{5}}
     \end{pspicture}
 \psset{unit=0.5cm}
  \begin{pspicture}(0,0)(8,4)
     \psline[linestyle=dashed](3,1)(5,1)

     \end{pspicture}
 \psset{unit=0.3cm}
  \begin{pspicture}(0,0)(8,4)
\pscustom{%
     \psline(5,2)(1.5,1.75)
     \gsave
     \psline(1.5,1.75)(4,4)
     \psline(5,2)(4,4)

     \fill[fillstyle=solid,fillcolor=lightgray]
\grestore}
     \psline(0,1)(1.5,0.25)
     \psline(1.5,0.25)(3,0)
     \psline(3,0)(5,0)
     \psline(5,0)(6.5,0.25)
     \psline(6.5,0.25)(8,1)
     \psline(8,1)(6.5,1.75)
     \psline(6.5,1.75)(5,2)
     \psline[linewidth=0.3pt,linestyle=dashed](5,2)(3,2)
     \psline[linewidth=0.3pt,linestyle=dashed](3,2)(1.5,1.75)
     \psline(1.5,1.75)(0,1)
     \psline(3,0)(1.5,1.75)
     \psline(1.5,0.25)(1.5,1.75)
     \psline(5,0)(1.5,1.75)
     \psline(6.5,0.25)(1.5,1.75)
     \psline(8,1)(1.5,1.75)
     \psline(6.5,1.75)(1.5,1.75)
     \psline(5,2)(1.5,1.75)
     \psline(5,2)(4,4)
     \psline[linewidth=0.3pt,linestyle=dashed](3,2)(4,4)
     \psline(1.5,1.75)(4,4)

     \put(4,-1){\mbox{6}}
     \end{pspicture}
\hspace*{0.5cm}
 \psset{unit=0.3cm}
  \begin{pspicture}(0,0)(8,4)
     \psline(0,1)(1.5,0.25)
     \psline(1.5,0.25)(3,0)
     \psline(3,0)(5,0)
     \psline(5,0)(6.5,0.25)
     \psline(6.5,0.25)(8,1)
     \psline(8,1)(6.5,1.75)
     \psline(6.5,1.75)(5,2)
     \psline(5,2)(3,2)
     \psline(3,2)(1.5,1.75)
     \psline(1.5,1.75)(0,1)
     \psline[linewidth=0.3pt](3,0)(1.5,1.75)
     \psline[linewidth=0.5pt](1.5,0.25)(1.5,1.75)
     \psline[linewidth=0.3pt](5,0)(1.5,1.75)
     \psline[linewidth=0.3pt](6.5,0.25)(1.5,1.75)
     \psline[linewidth=0.3pt](8,1)(1.5,1.75)
     \psline[linewidth=0.3pt](6.5,1.75)(1.5,1.75)
     \psline[linewidth=0.3pt](5,2)(1.5,1.75)
     \psline(0,1)(-0.25,0)
     \psline(0,1)(0,-0.15)
     \psline(1.5,0.25)(1,-0.5)
     \psline(1.5,0.25)(1.7,-0.75)
     \psline(3,0)(2.6,-1)
     \psline(3,0)(3.3,-1)
     \psline(5,0)(4.6,-1)
     \psline(5,0)(5.3,-1)
     \psline(6.5,0.25)(6.3,-0.75)
     \psline(6.5,0.25)(7,-0.5)
     \psline(8,1)(8.25,0)
     \psline(8,1)(8,-0.15)
     \psline[linewidth=0.3pt,linestyle=dashed](6.5,1.75)(6.25,1.25)
     \psline[linewidth=0.3pt,linestyle=dashed](5,2)(5.25,1.5)
     \psline[linewidth=0.3pt,linestyle=dashed](3,2)(3.25,1.5)
     \psline[linewidth=0.3pt,linestyle=dashed](1.5,1.75)(1.25,1.25)
     \psline[linewidth=0.3pt,linestyle=dashed](6.5,1.75)(6.75,1.25)
     \psline[linewidth=0.3pt,linestyle=dashed](5,2)(4.75,1.5)
     \psline[linewidth=0.3pt,linestyle=dashed](3,2)(2.75,1.5)
     \psline[linewidth=0.3pt,linestyle=dashed](1.5,1.75)(1.75,1.25)

     \psline(4,-3)(3,-2.75)
     \psline(4,-3)(3.5,-2.63)
     \psline(4,-3)(4,-2.5)
     \psline(4,-3)(4.5,-2.63)
     \psline(4,-3)(5,-2.75)

     \put(1,-1.5){\mbox{$\cdot$}}
     \put(7,-1.5){\mbox{$\cdot$}}
     \put(1.5,-2){\mbox{$\cdot$}}
     \put(6.5,-2){\mbox{$\cdot$}}
     \put(2,-2.3){\mbox{$\cdot$}}
     \put(6,-2.3){\mbox{$\cdot$}}

     \put(4,-4){\mbox{7}} 
     \end{pspicture}
 \end{center}

\vspace{0.5cm}

\end{proof}

\begin{rem}\rm
Let ${\mathfrak P}_n$ be the 3-uniform clutter as in Corollary~\ref{polytope}. Let $I$ be the circuit ideal of  $\bar{\mathfrak P}_n$ and $\Delta$ be a cimplicial complex such that the Stanley-Reisner ideal of $\Delta$ is $I$. In this case, $\Delta^{\vee}$, the Alexander dual of $\Delta$, is a pure simplicial complex of dimension $n-4$ which is not Cohen-Maculay, but adding any more facet to $\Delta^{\vee}$ makes it Cohen-Macaulay.
\end{rem}




\end{document}